\documentclass[11.6pt]{article}
\usepackage{latexsym,color,amsmath,amsthm,amssymb,amscd,amsfonts}
\usepackage{tikz}
\usetikzlibrary{arrows}
\usepackage{scalefnt}
\usepackage[font=small,labelfont=bf]{caption}
\usepackage{float}
\usepackage{graphicx} 
\usepackage{dsfont}
\usepackage{amsopn}
\usepackage{relsize}
\usepackage{mathrsfs}
\usepackage[colorlinks, citecolor=blue]{hyperref}
\usepackage{cleveref}
\usepackage{soul}
\usepackage{cancel}

\newtheoremstyle{nonum}{}{}{\itshape}{}{\bfseries}{.}{ }{\thmnote{#3}}

\newtheorem{thm}{Theorem}[section]
\newtheorem*{thm*}{Theorem}

\newtheorem{lem}[thm]{Lemma}
\newtheorem{prop}[thm]{Proposition}
\newtheorem{rem}[thm]{Remark}

\newtheorem{definition}[thm]{Definition}
\newtheorem*{definition*}{Definition}

\newtheorem*{rems*}{Remarks}

\theoremstyle{nonum}

\newcommand{\R}{\mathbb R}

\newcommand{\Z}{\mathbb Z}
\newcommand{\N}{\mathbb N}

\def\eps{{\varepsilon}}

\def\bdelta{ {\bar{\delta}} }

\newcommand{\tOmega}{\widetilde{\Omega}}

\renewcommand{\S}{\mathbb{S}}
\renewcommand{\H}{\mathbb{H}}

\DeclareFontFamily{OMX}{yhex}{}
\DeclareFontShape{OMX}{yhex}{m}{n}{<->yhcmex10}{}
\DeclareSymbolFont{yhlargesymbols}{OMX}{yhex}{m}{n}
\DeclareMathAccent{\wideparen}{\mathord}{yhlargesymbols}{"F3}

\begin{document}
	\title {Caustic-Free Regions for Billiards on\\ Surfaces of Constant Curvature}
	\date{}
	\author{D.I. Florentin, Y. Ostrover, D. Rosen}
	\maketitle
	\begin{abstract}
	In this note we study caustic-free regions for convex billiard tables in the hyperbolic plane or the hemisphere.
	In particular, following a result by Gutkin and Katok in the Euclidean case, we estimate the size of such regions in terms of the geometry of the billiard table. Moreover, we extend to this setting a theorem due to Hubacher which shows that no caustics exist near the boundary of a convex billiard table whose curvature is discontinuous.
	
	\end{abstract}

\section{Introduction and results}\label{Sec_Intro}

Starting with the pioneering work of Birkhoff~\cite{Birk}, billiard dynamics, which describes the motion of a massless particle
in a bounded domain with a perfectly reflecting boundary, has been extensively studied from various points of view (see e.g.,~\cite{KatHas, KozTre,Tab}).
An important role in understanding the dynamics of convex planar billiard tables is played by the existence, persistence,  and geometric and dynamical properties of caustics. Recall that a caustic is a curve inside the billiard table with the property that every billiard trajectory once tangent to it, remains tangent after every reflection at the boundary. It is known that there is a natural correspondence between caustics and invariant circles of the billiard map (see e.g., Chapter 2 of~\cite{Tab}). 
A classical result of Lazutkin~\cite{Laz} states that for a planar billiard table which is strictly convex and smooth enough 
 there exists an infinite collection of caustics close to the boundary of the table.
By contrast, Mather~\cite{Ma} showed that if the curvature of the boundary of the billiard table vanishes at one point, then the dynamics possesses no caustics at all. Moreover, Hubacher~\cite{Hu} showed that a discontinuity of the curvature excludes caustics from a neighborhood the boundary of the table.  In~\cite{GutKat}, Gutkin and Katok obtained a quantitative version of Mather's result, and 
provided estimates on the size of caustic-free regions for planar Euclidean billiards in terms of
the geometry of the billiard table (cf. Bialy~\cite{Bia,Bia1}).

In this note we study caustic-free 
regions for convex billiards in the hyperbolic plane
${\mathbb H}^2$ and in the hemisphere $\S^2_+$ (see Section~\ref{Sec_Prelim} below for the relevant definitions). Unless specifically stated otherwise, in what
follows we consider only convex caustics, as the methods
we use in the proofs require this assumption. Motivated by the work~\cite{GutKat} of Gutkin and Katok, our first result reads:


\begin{thm}\label{Thm_Main}
Let $K$ be a convex billiard table with $C^2$-smooth boundary in either $\H^2$ or $\S^2_+$. Assume that the boundary $\partial K$ has minimal curvature $\kappa_{\min}$, and that the table $K$ has diameter $D$. In the case of $\S^2_+$, assume further that $D < \frac{\pi}{2}$.
Let
$$ \varepsilon = \begin{cases} 
    \, {\rm arctanh} \left(\sqrt{2}\kappa_{\min} \tanh^{\frac{3}{2}}(D) \sinh^{\frac{1}{2}}(D)\right), & \text{in $\H^2$},\\
    2 \, {\rm arctan} \left(\frac{\sqrt{\pi}}{2}\kappa_{\min} \tan^{2}(D) \right)                      , & \text{in $\S^2_+$}.
	\end{cases}$$ 
Then, every convex caustic in $K$ lies in the $\varepsilon$-neighborhood of the boundary $\partial K$. 	
\end{thm}
Here, the $ \varepsilon$-neighborhood of $\partial K$ is the set $(\partial K)_{\varepsilon}$ of points whose distance to $\partial K$ is at most $\varepsilon$. 
Theorem~\ref{Thm_Main} states that $K\setminus (\partial K)_{\varepsilon}$ is a caustic-free region, i.e., a subset of the billiard table which no convex caustic may intersect. Note that when $\varepsilon$ exceeds the inradius of $K$, the neighborhood $K\setminus (\partial K)_{\varepsilon}$ is empty, and Theorem~\ref{Thm_Main} provides no caustic-free region. We remark that the restriction $D < \frac{\pi}{2}$ in the spherical case is technical, and a similar result may hold without it.

We note that when the billiard table has a flat point, i.e., where $\kappa_{\min}=0$, Theorem~\ref{Thm_Main} recovers the known result that there are no convex caustics in the interior of the table $K$ (cf. Mather's result~\cite{Ma}, and Section \textsection 6 of~\cite{GutTab} for the case of Minkowski billiards).

Note that, in contrast to the Euclidean case, in $\mathbb{H}^2$ the caustic-free region $K \setminus (\partial K)_\varepsilon$ may be disconnected, as shown in Figure \ref{fig:stadium}. For more detail, see Remark~\ref{rmk:hyperbolic-stadium} below. 

\begin{figure}[h]
    \centering
    \includegraphics[scale=0.8]{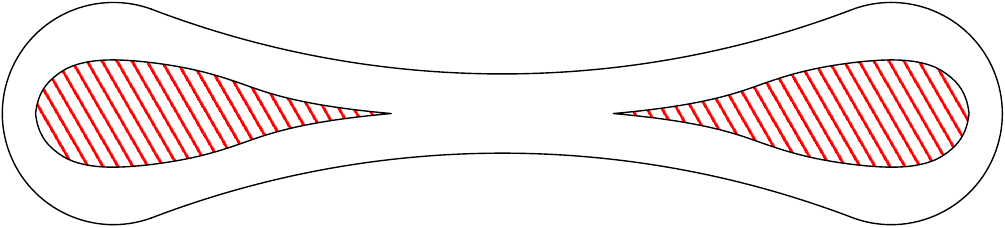}
    \caption{A disconnected caustic-free region for a billiard table in the hyperbolic plane}
    \label{fig:stadium}
\end{figure}

Our next result is an analogue  of Hubacher's Theorem~\cite{Hu} regarding the absence of caustics near the boundary of a billiard table.
More precisely,  
\begin{thm}\label{Thm_Hubbacher_Hyperbolic}
		Let $K$ be a convex billiard table in either ${\mathbb H}^2$ or $\S^2_+$ with $C^1$-smooth and piecewise $C^2$-smooth boundary. 
	Assume further that the curvature of $\partial K$ is continuous except for finitely many jump discontinuities, bounded away from zero, 
and has at least one discontinuity point.
		%
	Then the boundary $\partial K$ has a neighborhood which
	no caustic can intersect.  
	\end{thm}

\begin{rem} {\rm
As in the Euclidean case (cf.~\cite[Figure 3]{Hu}), an example of a convex billiard table $K$ in $\H^2$ (or $\S^2_+$) which satisfies the conditions of Theorem~\ref{Thm_Hubbacher_Hyperbolic} above may be obtained via the classical ``string-construction" around an equilateral triangle (see Section~\ref{Sec_Prelim} below for details, in particular Remark~\ref{rem-convexity-of-string-construction}). 
The boundary $\partial K$, which consists of six
pieces of hyperbolic (or spherical) ellipses, is globally $C^1$, but the curvature is discontinuous in
the six points where the ellipses are glued together.  
\textcolor{black}{Note, moreover, that this example demonstrates that under the assumptions of Theorem \ref{Thm_Hubbacher_Hyperbolic}, the billiard table may still possess convex caustics away from the boundary.
} 
}
\end{rem}

\noindent {\bf The paper is organized as follows:} in Section~\ref{Sec_Prelim} we recall some basic definitions and facts regarding billiard dynamics. 
In Section~\ref{sec-proof-main-thm} we prove Theorem~\ref{Thm_Main}, and in Section~\ref{sec-proof-of-hyp-Hubacher} we prove Theorem~\ref{Thm_Hubbacher_Hyperbolic}.

\noindent{\bf Notations:} 
In what follows, $\H^2$ denotes the hyperbolic plane, and $\S^2_+$ denotes an open hemisphere. The distance function on either manifold is denoted by $d(x,y)$, and 
the distance from a point $x$ to a set $A$ is given by $d(x,A)=\inf_{a\in A}\{d(x,a) \}$. A set $K$ is said to be convex if for every pair of points in $K$, the (unique) geodesic segment joining them is contained in $K$, and is said to be strictly convex if it does not contain
a geodesic segment in its boundary. The convex hull of two sets $A,B$ is denoted by ${\rm Conv}(A,B)$. 
The inradius of a convex set $K$ is the maximal radius of a disk contained in $K$. The diameter of a convex set $K$ is denoted by $D=\max\{d(x, y) \, | \, x,y \in \partial K \}$. The geodesic curvature of a regular curve $\gamma$ is denoted by $\kappa = \kappa_\gamma$. Finally, we denote by ${\rm Per(K)}$ the perimeter of the set $K$.

\noindent {\bf Acknowledgements:} We are grateful to Misha Bialy and Serge Tabachnikov for useful comments.
DIF was partially supported by the U.S. National Science Foundation Grant
DMS-1101636. YO is partially supported by the European
Research Council starting grant No. 637386, and by the ISF grant No. 667/18. DR is partially supported by the SFB/TRR 191 ‘Symplectic Structures in Geometry, Algebra and Dynamics’, funded by the DFG (Projektnummer 281071066 – TRR 191).

\section{Preliminaries}\label{Sec_Prelim}

Let $K$ be a convex set (in either ${\mathbb H}^2$ or ${\mathbb S}^2_+$) with $C^1$-smooth boundary $\partial K$.
As in the Euclidean case, the inner billiard (or Birkhoff billiard) in $K$, is the dynamical system corresponding to the free motion of a point particle inside $K$ (i.e. via geodesic lines), and reflecting elastically on impact with $\partial K$, making equal angles with the tangent line at the impact point.
The standard phase space of the billiard map is the cylinder $\partial K \times [0,\pi]$.
Set $s$ for the arclength parameter, $t \in [0,\pi]$ for the angle with the positive tangent,
and let $l$ be the length of $\partial K$.
The billiard map associated with $K$ is the map $\phi : {\mathbb R}/l{\mathbb Z} \times [0,\pi] \rightarrow  {\mathbb R}/l{\mathbb Z} \times [0,\pi]$ which sends a pair $(s_0, t_0)$ representing an impact, to the pair $(s_1, t_1) = \phi(s_0, t_0)$ corresponding to the next impact point (see Figure \ref{fig:billiard_map}). The map $\phi$ is well known to be an area-preserving monotone twist map, with generating map given by the distance function between the two consecutive billiard points (see, e.g.,~\cite{CP}, for a formal presentation of billiard dynamics on surfaces of constant curvature). We recall that the monotone twist condition implies in particular that \begin{equation} \label{eq-mon-twist-cond}
{\frac {\partial s_1} {\partial t_0} } > 0.
\end{equation}
For more details on monotone twist maps we refer the reader, e.g., to Chapter \textsection 1 of~\cite{Sib}. 
\begin{figure}[h]
    \centering
    \includegraphics[scale=0.3]{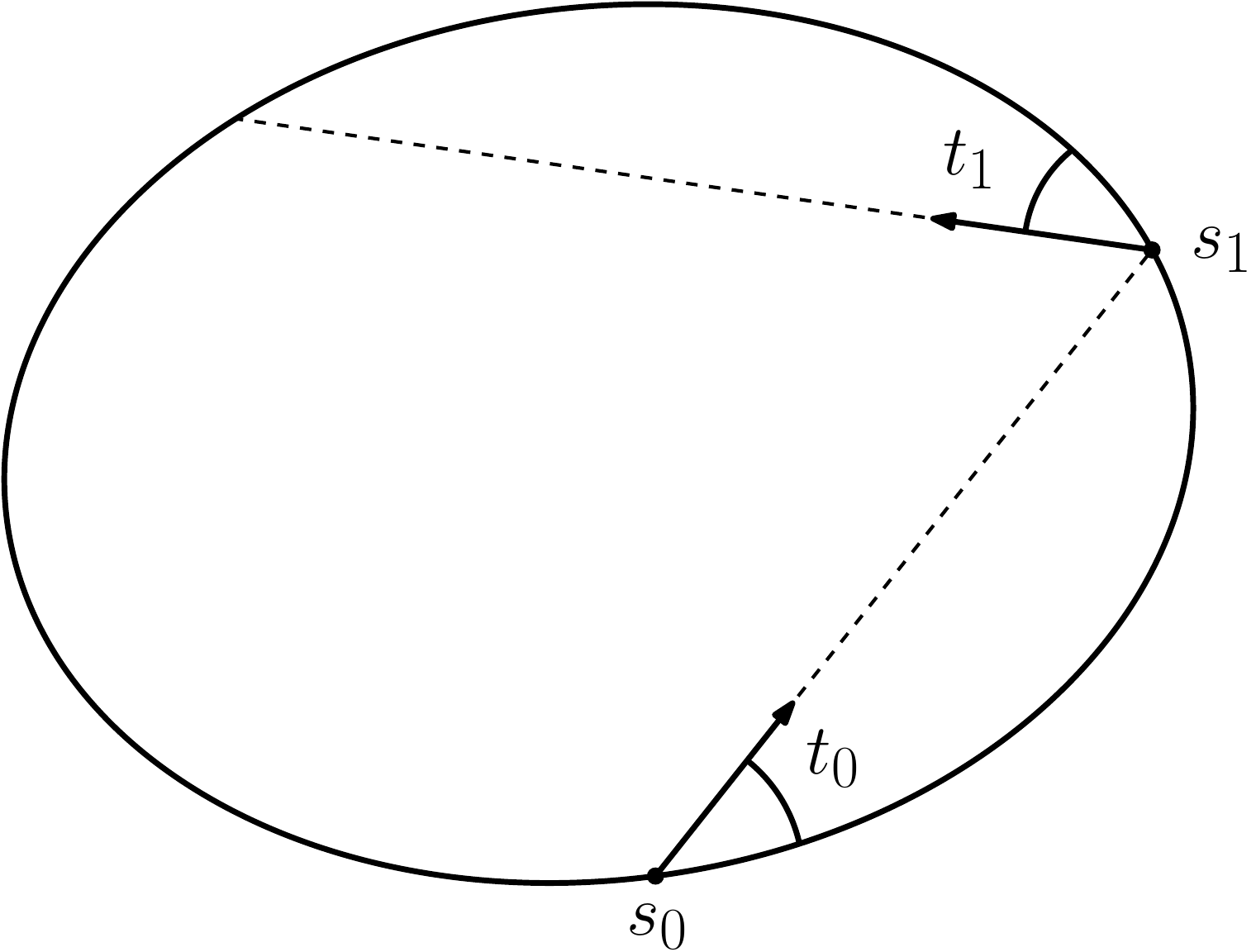}
    \caption{The billiard ball map}
    \label{fig:billiard_map}
\end{figure}


Caustics play an important role in the study of planar  billiards, and they are closely related with the geometry of the billiard table. In this note we consider only ``convex caustics". More precisely,

\begin{definition} \label{def_caustic}
A simple closed curve $\gamma \subset K$ is called a convex caustic if $\gamma$ bounds a convex set, and 
if any trajectory of the billiard flow tangent to $\gamma$, remains tangent after the reflection with $\partial K$.
\end{definition}

We remark that the notion of caustics for planar convex billiard tables is closely related with the notion
of an ‘invariant circle’ of the associated monotone twist map.
In particular, any convex caustic gives rise to such an invariant circle (see  e.g.,~\cite{GutKat, Tab}).

We recall next the classical ``string construction": to every convex set $C$ in the plane one can
associate a 1-parameter family of convex billiard tables $\{K_L\}$, where $L>0$, such that each table has
$C$ as a caustic. Roughly speaking, $\partial K_L$ is obtained by the following procedure: wrap a
loop of inelastic string of length $L$ around C. Then, pull the string tight away from $C$
to produce a point $p$ on the boundary of the billiard table $K_L$. Finally, move the point $p$
around $C$, keeping the string tight, to obtain the rest of $\partial K_L$.
Note that this string construction, which was originally studied in the context of Euclidean geometry (see~\cite{Sto, Tur}), can be naturally generalized to surfaces of constant curvature (see~\cite{Glut}, and Section \textsection 3 of~\cite{GutTab} for the more general setting of Finsler billiards). More precisely, given a convex set $C$ (in either ${\mathbb H}^2$ or ${\mathbb S}^2_+$) and $L >0$, we set 
$$K_L = \{ q  \ | \ {\rm Per}  \left ({\rm Conv}(q,C) \right) \leq   {\rm Per}(C) + L \}.$$
It is known  that $K_L$ is a  billiard table for which $C$ is a convex caustic\footnote{In ${\mathbb S}^2_+$, $L$ must be sufficiently small in order for $K_L$ to be a convex set whose boundary $\partial K_L$ is a closed curve.}, and conversely, for any caustic $C$ in a convex billiard table $K$ the function ${\rm Per}  \left ({\rm Conv}(\cdot ,C) \right) $ on $\partial K$ is constant  (see, e.g., Lemma 3.6 in~\cite{GutTab}). 
The numerical value $L$ above is called the Lazutkin parameter of the caustic $C$ (see, e.g.~\cite{GutKat}).  
\begin{rem} \label{rem-convexity-of-string-construction}
{\rm Note that in $\mathbb H^2$, the set $K_L$ is convex for every $L>0$. Indeed, when $C$ is a segment, the convexity of the
``hyperbolic ellipse" $K_L$ follows from the convexity of the hyperbolic distance
function (see e.g., Theorem 2.5.8 in~\cite{Thur}). When $C$ is a convex polygon, $K_L$
is convex since its boundary is obtained by gluing a finite number of hyperbolic
elliptical arcs, and the normal to $\partial K_L$ is continuous at the gluing
points. 
A standard approximation argument implies the general case.}
\end{rem}

The following mirror equation for billiards in $\H^2$ and $\S^2_+$ can be found, e.g., in~\cite{Bia,GutSmiGut}.
Denote by $a$ and $b$ the lengths of the two tangent lines from a point $m \in \partial K$ to the caustic $C$, and let $\theta$ be the angle between either of these lines and $\partial K$ (see Figure \ref{fig:mirror}). 
	\begin{equation}\label{eq-mirror}
	\begin{cases}
	\begin{aligned}
	\displaystyle \frac{1}{\tanh (a)} &+ \frac{1}{\tanh (b)}&= \frac{2\kappa(m)}{\sin(\theta)}, \quad & \text{in $\H^2$}, \\
    \displaystyle \frac{1}{\tan (a)} &+ \frac{1}{\tan (b)}&= \frac{2\kappa(m)}{\sin(\theta)}, \quad & \text{in $\S^2_+$}.   
	\end{aligned}
	\end{cases}
	\end{equation}
Here $\kappa(m)$ stands for the curvature of $\partial K$ at the point $m$.

\begin{figure}[h]
    \centering
   \includegraphics[scale=0.7]{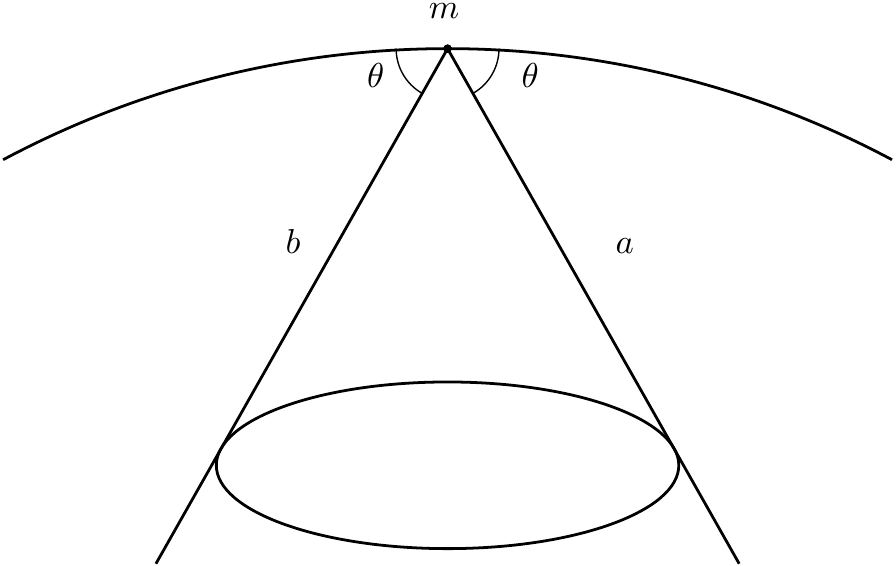}
    \caption{The mirror equation}
    \label{fig:mirror}
\end{figure}

\section{Caustic-free regions away from the boundary} \label{sec-proof-main-thm}

In this section we prove Theorem~\ref{Thm_Main}. The proofs for $\mathbb H^2$ and $\mathbb S^2_+$ are very similar, and we provide full details only for the hyperbolic case. We refer the reader to Remark \ref{rem_spherical-3.3} and Remark \ref{rem_spherical-3.4} for the adjustments in the spherical case.
Let $K \subset {\mathbb H}^2$ be a convex billiard table with
$C^2$-smooth boundary. 
Let $r = r(K)$ be the inradius of $K$, and 
consider the function $$\delta : K \rightarrow [0,r(K)],  \ \ {\rm given \ by \ }  \delta (x) = d(x,\partial K).$$ 
We provide an upper bound on the value that $\delta$ may attain on a convex caustic. For that purpose, given a convex caustic $\gamma \subset K$, we denote 
by $\delta_{\gamma}$ the maximal
distance from $\gamma$ to $\partial K$, i.e.,
	\begin{equation*}
		\delta_{\gamma}
		=\max_{x\in \gamma} \delta(x).
	\end{equation*}
The main ingredient in the proof of Theorem~\ref{Thm_Main} is the following upper bound of $\delta_{\gamma}$ in terms of the diameter $D = D(K)$ and the minimal
	curvature  $\kappa_{\min} = \kappa_{\min}(K)$  of the billiard table $K$.
	\begin{prop}\label{Prop_Main}
	Let $K\subset {\mathbb H}^2$ be a $C^2$-smooth convex billiard table, and let $\gamma \subset K$ be a convex caustic. Then,
		\begin{equation} \label{Ineq_delta-gamma}
			\tanh \delta_{\gamma} <
			\sqrt{2}\kappa_{\min} \tanh^{\frac{3}{2}}(D) \sinh^{\frac{1}{2}}(D).
		\end{equation}
	\end{prop}
Before proving Proposition~\ref{Prop_Main}, let us first formally deduce Theorem \ref{Thm_Main} from it.

\noindent {\bf Proof of Theorem \ref{Thm_Main}.}
Note that if $\gamma$ is a convex caustic, then
by \eqref{Ineq_delta-gamma}, for every $x \in \gamma$ one has 
	$$ \delta(x) \leq \delta_{\gamma} < 
	{\rm arctanh}
	\left(
	\sqrt{2}\kappa_{\min} \tanh^{\frac{3}{2}}(D) \sinh^{\frac{1}{2}}(D)
	\right) = \varepsilon.$$
	Thus, $x \in  (\partial K)_{\varepsilon}$, and hence $\gamma \subseteq (\partial K)_{\varepsilon}$ as required. 
	\qed
	
	\begin{rem}\label{rmk:hyperbolic-stadium}{\rm
As noted above, in $\mathbb{H}^2$ it may happen that the caustic-free region $K \setminus (\partial K)_\eps$ is disconnected (see Figure \ref{fig:stadium}). Consider two hyperbolic disks of radius $R$, and their convex hull $S$ (this is a hyperbolic analog of the classical ``stadium'' billiard table). If the distance between the disks is sufficiently large, the minimal width of $S$ (in the sense of Santal\'{o} \cite{Santalo-width}, i.e., the minimal projection of $S$ on a geodesic line normal to its boundary) is arbitrarily small. The table $K$ is a strictly convex approximation of $S$.
This is an adaptation of the following observation due to Badt \cite{Badt}. In the hyperbolic plane, the inradius of a convex domain is not a lower bound for the minimal width. This shows in particular that the minimal width is not monotone with respect to inclusion.
}
\end{rem}



The idea behind the proof of Proposition~\ref{Prop_Main} is the following. First, we use the fact that for a convex caustic $\gamma \subset K$, the billiard table $K$ can be obtained from $\gamma$ by a string construction (see Section~\ref{Sec_Prelim} above). We recall that the outcome of a string construction (with different string lengths) is a family of billiard tables that are parameterized by the Lazutkin parameter $L$. Proposition~\ref{Prop_Main} is proven by comparing both sides of inequality~$(\ref{Ineq_delta-gamma})$ with the Lazutkin parameter associated with the caustic $\gamma$, using the mirror equation~$(\ref{eq-mirror})$, hyperbolic trigonometry, and some other geometric features of the Lazutkin parameter $L$. We divide the argument into two lemmas: 
	%
	
	\begin{lem}\label{Prop_Low-Bd}
		Let $K\subset {\mathbb H}^2$ be a $C^2$-smooth convex billiard table, and $\gamma \subset K$ be a convex caustic with Lazutkin parameter $L$. Then, 
		\[
		%
		%
		\frac{\tanh^2 (\delta_{\gamma})}{ \tanh (D)} < L.
		\]
	\end{lem}
	
	\begin{lem}\label{Prop_Up-Bd}
		Let $K\subset {\mathbb H}^2$ be a $C^2$-smooth convex billiard table, and $\gamma \subset K$ be a convex caustic with Lazutkin parameter $L$. Then,
		$$
		L < 2 \kappa_{\min}^2 \tanh^2 (D) \sinh (D).
		$$
	\end{lem}
	\noindent Combining Lemma~\ref{Prop_Low-Bd} and Lemma~\ref{Prop_Up-Bd} one immediately obtains Proposition \ref{Prop_Main}, and hence Theorem \ref{Thm_Main}, in the hyperbolic case. In the spherical case, the analogous results are Equations \eqref{eq_spherical-3.3} and \eqref{eq_spherical-3.4}, which read, respectively,
	\begin{equation*}
	\frac{4 \tan^2\left( \frac{\delta_\gamma}{2} \right)}{\tan (D)} < L \qquad \text{and}\qquad L < \pi \kappa_{\min}^2 \tan^3(D),
	\end{equation*}
	see Remarks \ref{rem_spherical-3.3} and \ref{rem_spherical-3.4}, respectively. From these, one immediately obtains
	\begin{equation*}
	 	 \tan\left( \frac{\delta_\gamma}{2} \right) < \frac{\sqrt{\pi}}{2} \kappa_{\min} \tan^2(D),
	\end{equation*}
	which is the spherical analogue of Proposition \ref{Prop_Main}. The proof of Theorem \ref{Thm_Main} in the spherical case now follows in the same manner as detailed above.
	
	\subsection{Proof of Lemma~\ref{Prop_Low-Bd}} \label{Sec_Low-Bound}
	
For the proof of Lemma \ref{Prop_Low-Bd} we introduce the auxiliary parameter
	$$
	\bdelta_{\gamma} := \max_{m\in \partial K} d(m,\gamma).
	$$
	We remark that $\bdelta_{\gamma}$ is simply the hyperbolic Hausdorff distance between ${\rm Conv}(\gamma)$ and $K$. 
	
	\begin{lem}\label{Lem_Comp-Deltas}
		For any convex caustic $\gamma$ in $K$ one has $\bdelta_{\gamma} \geq \delta_{\gamma}$.
	\end{lem}
	
	\begin{proof}[{\bf Proof of Lemma~\ref{Lem_Comp-Deltas}}]
		First, note that for any point $y \in \gamma$ there is $m \in \partial K$ with $d(m,y)=d(m,\gamma)$. Indeed, let $l$ be a geodesic ray normal to $\gamma$ at $y$, pointing outwards.
		Denote by $m \in \partial K$ the intersection of $l$ with $\partial K$ (here we use the fact that $\gamma \subset K$), and put $r=d(m,y)$. The closed disk $B = B_{r,m}$ of radius $r$ about $m$ intersects $\gamma$ at $y$, and $\partial B$ is tangent to $\gamma$ at $y$. As the disk $B$ is strictly convex, it follows that $B \cap \gamma = \{y\}$, and hence $$d(m,\gamma) = r = d(m,y).$$
Therefore, for any $y \in \gamma$ we take $m\in \partial K$ as above, and get
		$$
		d(y, \partial K) \leq d(y,m) = d(m,\gamma) \leq \bdelta_{\gamma}.
		$$
		Maximizing over $y \in \gamma$ gives $\delta_{\gamma} \leq \bdelta_{\gamma}$, as claimed (see Figure \ref{fig:delta}).
		\end{proof}
		\begin{figure}[h]
		    \centering
		    \includegraphics[width=0.4\textwidth]{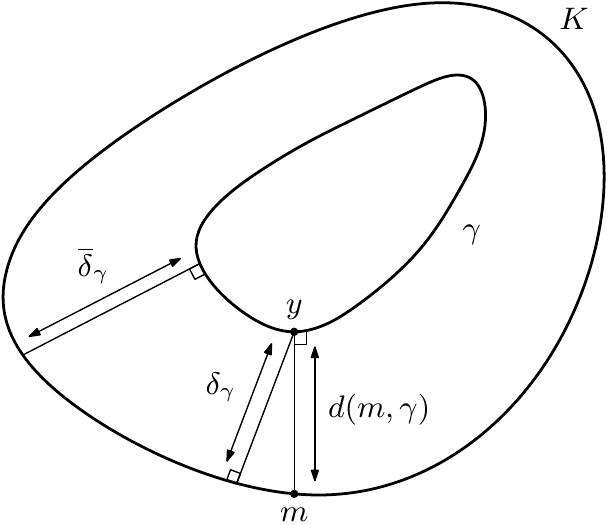}
		    \caption{The inequality $\delta_\gamma = \max_{y \in \gamma}d(y,\partial K) \leq {\bdelta}_\gamma$ for the maximizing point $y \in \gamma$}
		    \label{fig:delta}
		\end{figure}
	
	%
	%
	%
	
	\noindent {\bf Proof of Lemma \ref{Prop_Low-Bd}.}
	In view of Lemma \ref{Lem_Comp-Deltas}, it suffices to prove the inequality
	\begin{equation*}
		%
		%
		\frac{\tanh^2(\bdelta_{\gamma})}{\tanh(D)}
		<
		L.
	\end{equation*}
	Let $m \in \partial K$ and $y \in \gamma$ such that 
	$$
	d(m,y) = d(m, \gamma) = \bdelta_{\gamma}.
	$$
	Note that the geodesic segment $[m,y]$ between $m$ and $y$ is normal
	to the caustic $\gamma$ at $y$, since $y$ is the minimizer of the function
	$\gamma \ni z \mapsto d(m,z)$. Let $P_1, P_2$ be the end points of
	the two tangents from $m$ to $\gamma$, and denote by $Q_1, Q_2$ the
	intersection of these two tangents with the geodesic line normal to
	$[m,y]$ at $y$ (see Figure~\ref{Fig_Low-Bound}).
	
	\begin{figure}[h]
		\centering
		\includegraphics[width=0.4\textwidth]{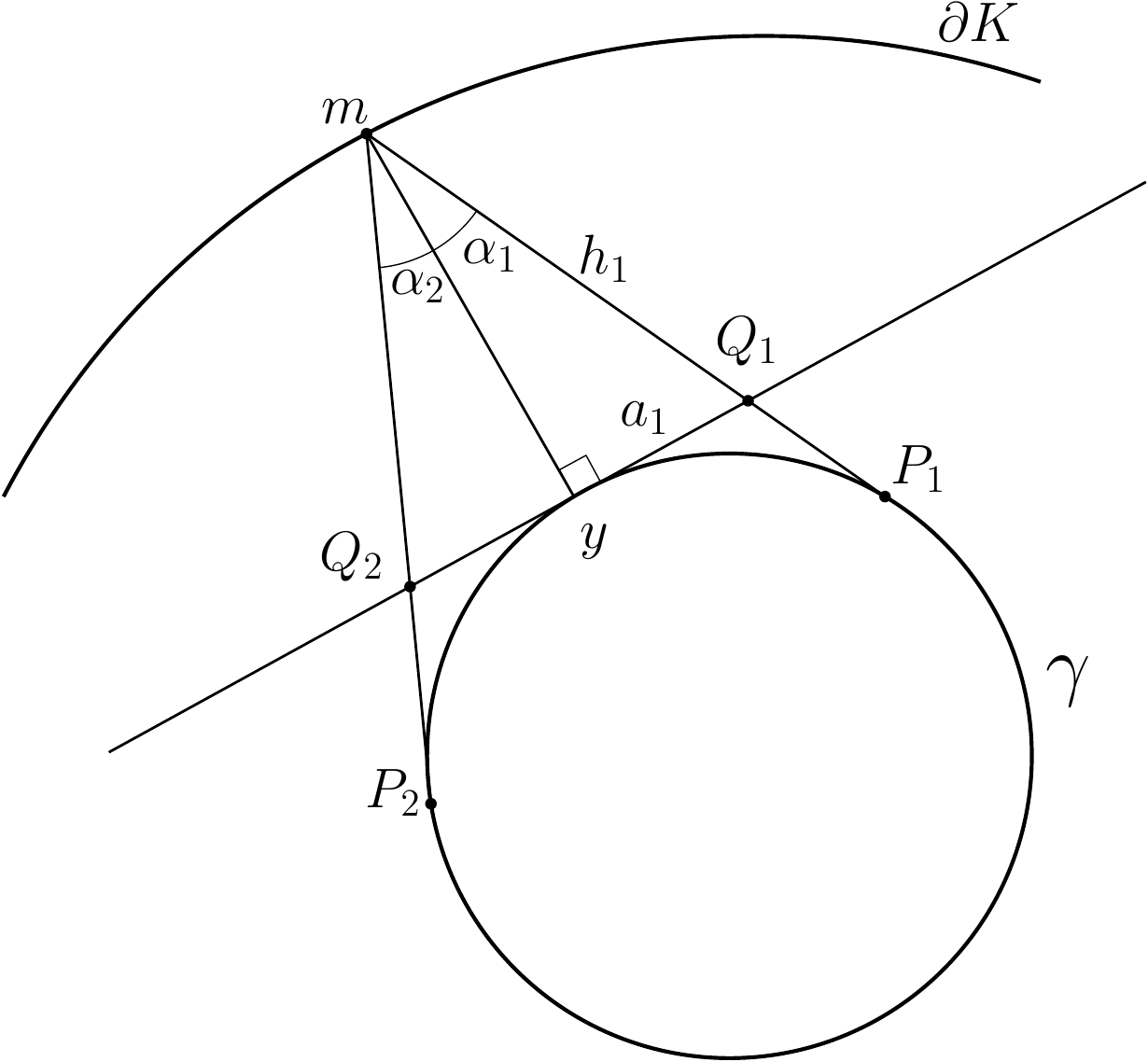}
		\caption{Notations for Lemma \ref{Prop_Low-Bd}.}
		\label{Fig_Low-Bound}
	\end{figure}
	
	Recall that it follows from the Crofton formula \cite[Section 3]{Santalo}
	that the perimeter of convex bodies in the hyperbolic plane is monotone with
	respect to inclusion, and thus:
	\begin{equation*}
		|P_1Q_1| + |Q_1Q_2| + |Q_2P_2| \ge |\wideparen{P_1yP_2}|.
	\end{equation*}
	Substituting this into the definition of $L$ gives
	\begin{eqnarray*}
		L &=& |mP_1| + |mP_2| - |\wideparen{P_1yP_2}| \ge
		|mP_1| + |mP_2| - |P_1Q_1| - |Q_1Q_2| - |Q_2P_2| = \\
		&=&|mQ_1| + |mQ_2| - |Q_1Q_2| =
		\left(|mQ_1| - |yQ_1|\right) +
		\left(|mQ_2| - |yQ_2|\right) = \\
		&=& \left(h_1-a_1\right) + \left(h_2-a_2\right),
	\end{eqnarray*}
	where 
	%
	$h_i=|mQ_i|,\, a_i=|yQ_i|$, and $\bdelta_{\gamma}=|my|$
	are the edge lengths of the triangles $\Delta myQ_i$. We denote the angles
	$\sphericalangle ymQ_i$ by $\alpha_i$ (note that $\alpha_i < \frac{\pi}{2}$) and
	define $\theta_i=\frac{\pi}{2} - \alpha_i$. Recall the hyperbolic laws of sine
	and cosine in a right triangle:
	\begin{eqnarray*}
		\cos(\theta_i) = \sin(\alpha_i) = \frac{\sinh(a_i)}{\sinh(h_i)} ,\\
		\sin(\theta_i) = \cos(\alpha_i) = \frac{\tanh({\bdelta_{\gamma}})}{\tanh(h_i)} .
	\end{eqnarray*}
	Note that there exists $z\in(a_1,h_1)$ such that
	$$\frac{\sinh(h_1)-\sinh(a_1)}{h_1-a_1}=\cosh(z) < \cosh(h_1),$$ and hence
	\begin{eqnarray*}
		h_1 - a_1 &>&
		\frac{\sinh(h_1)-\sinh(a_1)}{\cosh(h_1)} =
		\tanh(h_1)(1-\cos(\theta_1)) =\\ &=&
		\tanh({\bdelta_{\gamma}})\frac{1-\cos(\theta_1)}{\sin(\theta_1)} =
		\tanh({\bdelta_{\gamma}})\tan\left(\frac{\theta_1}{2}\right).
	\end{eqnarray*}
Without loss of generality we can assume that $\theta_1\le\theta_2$, hence
	\[
	L> \tanh({\bdelta_{\delta}})
	\left(
	\tan\left(\frac{\theta_1}{2}\right) + 
	\tan\left(\frac{\theta_2}{2}\right)
	\right) \ge
	2 \tanh({\bdelta_{\gamma}}) \tan\left(\frac{\theta_1}{2}\right),
	\]
	which we rewrite as
	\begin{equation}\label{Ineq_delta-gamma-1st}
		\tanh({\bdelta_{\gamma}})< \frac{L}{2\tan\left(\frac{\theta_1}{2}\right)}.
	\end{equation}
	On the other hand, since $h_1 < D$, 
	\begin{equation}\label{Ineq_delta-gamma-2nd}
		\tanh(\bdelta_{\gamma})=\sin(\theta_1)\tanh(h_1)<
		\frac
		{2\tan\left(\frac{\theta_1}{2}\right)}
		{1+\tan^2\left(\frac{\theta_1}{2}\right)}
		\tanh(D).
	\end{equation}
	Inequalities \eqref{Ineq_delta-gamma-1st} and
	\eqref{Ineq_delta-gamma-2nd} may be combined in order to remove the
	dependence on $\theta_1$. Define $$f(t):=\min\left\{\frac{L}{2t},\frac{2\tanh(D)t}{1+t^2}
	\right\}.$$ Note that for $s=\tan\left(\frac{\theta_1}{2}\right)
	\in[0,1]$ one has
	\[
	\tanh(\bdelta_{\gamma}) <
	f(s) \le
	\max_{0\leq t \leq 1}f(t).
	\]
	Since $f$ is the minimum of two functions, one increasing and one
	decreasing, two cases need to be considered in order to find $\max f$
	in $[0,1]$. Note first that if $\tanh(D) \leq {\frac L 2}$, then:
  $$\max_{0\le t \le 1}f(t)=\tanh(D)\le\frac{L}{2},$$
	and hence,
		\begin{equation*}\label{Ineq_L-Low-Bound-Case1}
			\frac{\tanh^2(\bdelta_{\gamma})}{\tanh(D)}
			<
			\tanh(D)
			\le
			\frac{L}{2}
			<
			L.
		\end{equation*}
On the other hand, if $\tanh(D) > {\frac L 2}$, then	$$\max_{0\leq t \leq 1}f(t)=\sqrt{L(\tanh(D)-L/4)}
		<\sqrt{L \tanh(D)}, $$
	 which again implies that
		\begin{equation*}\label{Ineq_L-Low-Bound-Case2}
			\frac{\tanh^2(\bdelta_{\gamma})}{\tanh(D)}
			<
			L,
		\end{equation*}
and the proof of the lemma is now complete.	
	\qed

\begin{rem}\label{rem_spherical-3.3} {\rm
In the spherical case, the analog of the statement in Lemma~\ref{Prop_Low-Bd} is:
\begin{equation}\label{eq_spherical-3.3}
\frac{4 \tan^2\left( \frac{\delta_\gamma}{2} \right)}{\tan (D)}\leq \frac{4 \tan^2\left( \frac{\bdelta_\gamma}{2} \right)}{\tan (D)} < L.
\end{equation}
Indeed, using the same notations as in Lemma~\ref{Prop_Low-Bd} (see figure \ref{Fig_Low-Bound}), one has:
\[
h_1 - a_1 
> \frac{\cos a_1 - \cos h_1} {\sin h_1} =
\frac{1}{\tan h_1} \left(\frac{\cos a_1}{\cos h_1} - 1 \right) \stackrel{(*)}{=}
\frac{\cos \alpha}{\tan \bdelta_\gamma} \left(\frac{1}{\cos \bdelta_\gamma} - 1 \right) =
\cos \alpha \tan \left( \frac{\bdelta_\gamma}{2} \right),
\]
where $(*)$ follows from the fact that $\tan \bdelta_\gamma = \sin\theta \tan h_1$, combined with the spherical Pythagoras Theorem $\cos h_1 = \cos \bdelta_\gamma \cos a_1$. Hence,
\[
\cos \alpha = \sin \theta < 
\frac {L} {2 \tan \left( \frac{\bdelta_\gamma}{2} \right)}.
\]
Finally, we conclude \eqref{eq_spherical-3.3} by combining the above inequalities, i.e., 
\[
\tan \left( \frac{\bdelta_\gamma}{2} \right) <
\frac12 \tan \bdelta_\gamma =
\frac12 \sin \theta \tan h_1 <
\frac12 \sin \theta \tan D <
 \frac {L \tan D} {4 \tan \left( \frac{\bdelta_\gamma}{2} \right)}.
\]
}
\end{rem}

\subsection{Proof of Lemma~\ref{Prop_Up-Bd}}\label{Sec_Upp-Bound}

In this section we provide an upper bound for the Lazutkin parameter $L$ of a convex caustic in terms of the diameter and the minimal curvature of the corresponding billiard table. 

%
%
	
	
	\noindent {\bf Proof of Lemma \ref{Prop_Up-Bd}.}
	Let $m \in \partial K$, and let $A$ and $B$ be the endpoints of
	the two tangents from $m$ to $\gamma$. Denote by $\theta$ the angle
	of reflection at $m$, and by $a,b,c$ the sides of the triangle $\Delta BmA$
	(see Figure \ref{fig-upper-bound}).
	Since $|AB| \le |\wideparen{AB}|$, we have $L \le |mA|+|mB|-|AB| =
	a + b - c$. The hyperbolic law of cosines in the triangle $\Delta BmA$ reads
	$$\cosh(c) = \cosh(a)\cosh(b) - \sinh(a)\sinh(b)\cos(\alpha).$$
	\begin{figure}[h]
		\centering
		\includegraphics[width=0.5\textwidth]{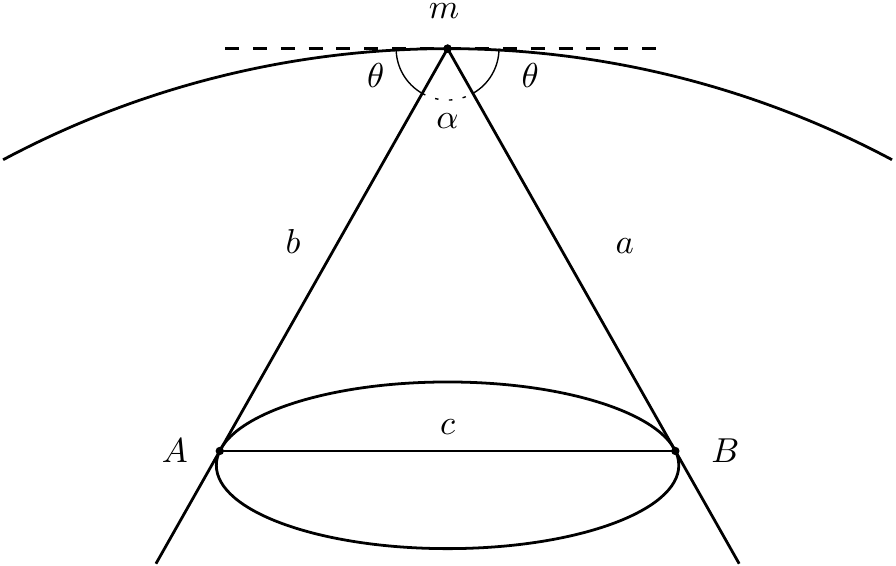}
		\caption{Notations for Lemma \ref{Prop_Up-Bd}.}
		\label{fig-upper-bound}
	\end{figure}

Thus by the  angle-sum formula for hyperbolic cosine one has
	%
	\begin{equation} \label{Eq_Hyp-Law-Cos}
		\cosh(a+b) - \cosh(c)  = 
		\sinh(a)\sinh(b) (1 + \cos \alpha ) =
		2 \sinh(a) \sinh(b) \sin^2 \theta.
		\end{equation}
On the other hand, note that, for $0<x<y$,
	\begin{equation}\label{Ineq_Jensen-lemma-for-cosh}
		\frac{\cosh y - \cosh x}{y - x} = {\frac {2 \sinh({\frac {y+x} {2} }) \sinh({\frac {y-x} {2} } ) } {y-x} }>
		\sinh\left(\frac{x + y}{2}\right)>
		\sinh\left(\frac{y}{2}\right).
	\end{equation}
Using \eqref{Eq_Hyp-Law-Cos}, \eqref{Ineq_Jensen-lemma-for-cosh} for $x = c < a + b = y$, and the hyperbolic mirror equation \eqref{eq-mirror} written in the form
	\begin{equation*}
		\sin(\theta) = 2\kappa(m) {\frac {\sinh(a)\sinh(b)} {\sinh(a+b)} },
	\end{equation*}
	one obtains
		%
	\begin{equation}\label{eq-up-last}
	\begin{aligned}
	L & \le a + b - c <
		\frac{\cosh(a+b) - \cosh c}{\sinh\left(\frac{a + b}{2}\right)} =
		\frac{2\sinh(a)\sinh(b)}{\sinh\left(\frac{a + b}{2}\right)}
		\sin^2(\theta) 
		 \\ & =
		\frac{2 \sinh (a)  \sinh (b)}{\sinh\left(\frac{a + b}{2}\right)}
		\left(2\kappa(m) \frac{\sinh(a)\sinh(b)}{\sinh(a+b)}\right)^2 
		%
		%
=		\frac{2 \kappa^2(m)}{\cosh^2 \left(\frac{a+b}{2}\right)}
		\left(
		\frac{\sinh(a)\sinh(b)}{\sinh\left(\frac{a+b}{2}\right)}
		\right)^3.
	\end{aligned}
	\end{equation}
	Since the function $\sinh:\R^+\to\R^+$ is log-concave, one has:
	$$
	\sinh \left(\frac{a+b}{2}\right) \ge
	\frac{\sinh(a)\sinh(b)}{\sinh\left(\frac{a+b}{2}\right)}.
	$$
	Plugging this into \eqref{eq-up-last}, we get:
	$$
	L < 2 \kappa^2(m) \,
	\frac{\sinh^3\left(\frac{a+b}{2}\right)}
	{\cosh^2\left(\frac{a+b}{2}\right)} \le
	2\kappa^2(m) \tanh^2(D) \sinh(D).
	$$
	Minimizing over $m \in \partial K$ yields the result.
	\qed

\begin{rem}\label{rem_spherical-3.4} {\rm
In the spherical case, the analog of the statement in Lemma~\ref{Prop_Up-Bd} is:
\begin{equation}\label{eq_spherical-3.4}
L < \pi \kappa_{\min}^2 \tan^3(D).
\end{equation}
Indeed, using the same notations as in the proof of Lemma~\ref{Prop_Up-Bd} (see Figure \ref{fig-upper-bound}),  by our assumption that  $\frac{c}{2}<\frac{a+b}{2} < D < \frac{\pi}{2}$, one has
\[
\sin\left(\frac{a+b+c}{2}\right) \ge \sin \left(\frac{a+b}{2}\right)\cos\left(\frac{c}{2}\right)>
\sin \left(\frac{a+b}{2}\right)\cos\left(\frac{D}{2}\right),
\]
and thus
\[
\frac{\cos(c) - \cos(a+b)}{a+b-c} =
\frac
{\sin \left(\frac{a+b-c}{2}\right) \sin \left(\frac{a+b+c}{2}\right)}
{\left(\frac{a+b-c}{2}\right)} >
\frac{2}{\pi} \sin \left(\frac{a+b}{2}\right)\cos\left(\frac{D}{2}\right).
\]
By the spherical law of cosine one has:
\[
\cos(c) - \cos(a+b) = 2 \sin(a) \sin(b) \sin^2(\theta).
\]
As in \eqref{eq-up-last}, by combining the above with the mirror equation \eqref{eq-mirror}, one concludes \eqref{eq_spherical-3.4} by
\begin{eqnarray*}
L &<& a + b - c <
\frac
{\cos(c) - \cos(a+b)}
{\frac{2}{\pi} \sin \left(\frac{a+b}{2}\right)\cos\left(\frac{D}{2}\right)} =
\pi \frac
{\sin(a) \sin(b) \sin^2(\theta)}
{\sin \left(\frac{a+b}{2}\right)\cos\left(\frac{D}{2}\right)} =
\frac
{4\pi \kappa_{\min}^2 \sin^3(a) \sin^3(b)}
{\sin \left(\frac{a+b}{2}\right) \sin^2(a+b) \cos\left(\frac{D}{2}\right)} 
\\ \\&\le&
\frac
{\pi \kappa_{\min}^2 \sin^6 \left(\frac{a+b}{2}\right)}
{\sin^3 \left(\frac{a+b}{2}\right) \cos^2 \left(\frac{a+b}{2}\right) \cos\left(D\right)} =
\frac
{\pi \kappa_{\min}^2 \tan^2 \left(\frac{a+b}{2}\right) \sin \left(\frac{a+b}{2}\right)}
{\cos\left(D\right)} < 
\pi \kappa_{\min}^2 \tan^3 \left( D \right).
\end{eqnarray*}
}
\end{rem}

\section{Caustic-free regions near the boundary} \label{sec-proof-of-hyp-Hubacher}

In this section we prove Theorem~\ref{Thm_Hubbacher_Hyperbolic}.
The proof follows the same lines as~\cite{Hu}. 
The geometry of the ambient space only plays a role in the proof of  Proposition~\ref{prop-the-big-one} (see Remark \ref{rem-spheric-fig-four} below). The other components in the proof of the two cases ($\mathbb H^2$ and $\mathbb S^2_{+}$) are essentially identical, thus throughout this section we restrict our attention to the hyperbolic case only.


Let $K \subset {\mathbb H}^2$ be a convex billiard table, with piecewise $C^2$-smooth boundary. 
Recall from Section~\ref{Sec_Prelim} above that the phase space of the billiard map is the cylinder $\Omega: =\R / l\Z \times [0,\pi]$ (where $l$ is the perimeter of the billiard table $K$). 
Here, by a slight abuse of notation, we use the coordinates $(s,t)$ both for the phase space and for its universal cover $\tOmega = \R \times [0,\pi]$ (where $s$ is the arclength parameter, and $t$ the angle to the tangent at $s$). A pair $(s_0, t_0)$ representing an impact is mapped to the pair $(s_1, t_1) = \phi(s_0, t_0)$ corresponding to the next impact point.
Recall additionally from Section \ref{Sec_Prelim} that the billiard map is an area-preserving monotone twist map.

We recall that an invariant circle $\Gamma$ is a curve in $\Omega$ that is homotopic to
one of the boundary components of $\Omega$, and such that $\phi (\Gamma) = \Gamma$.  
By Birkhoff's theorem (see, e.g., \cite{KatHas}), any invariant circle is a graph of a Lipschitz function $\R / l \Z \to [0,\pi]$, and moreover the Lipschitz constants of all such function are uniformly bounded.
Every convex caustic $\gamma$ in $K$ gives rise to an invariant circle of the billiard map, by considering the field of tangent vectors along $\partial K$ which point in the direction of positive tangency with $\gamma$. In particular, the boundary $\partial K$ corresponds to the trivial invariant circle $\R / l \Z \times \{0\} \subset \Omega$. 

To prove Theorem~\ref{Thm_Hubbacher_Hyperbolic} it suffices to prove that under its hypotheses there is a neighbourhood of the boundary circle $\R / l \Z \times \{0\}$ in the phase space $\Omega$ through which no other invariant circle can pass. The proof is divided into three parts. First, we show that the phase space contains a region of the form $I \times [0,\varepsilon)$ which is free of invariant circles, for some interval $I$ (see Proposition~\ref{prop-the-big-one}). Next, assuming the conclusion of Theorem~\ref{Thm_Hubbacher_Hyperbolic} is false, 
a standard limiting argument implies 
the existence of a non-trivial invariant circle which intersects the boundary circle $\R / l \Z  \times \{0\}$, and avoids the region $I \times [0,\varepsilon)$.  Finally,
we show that the existence of such an invariant circle is forbidden (Lemma~\ref{lem-step2}), and conclude that a caustic-free neighborhood of the boundary exists. The two main ingredients in the proof of Theorem~\ref{Thm_Hubbacher_Hyperbolic} are thus the following two claims.  

\begin{lem}\label{lem-step2}
Let $K \subset {\mathbb H}^2$ be a convex set with boundary $\partial K$ which is piecewise $C^2$-smooth. 
If $\Gamma$ is a non-trivial invariant circle then it is disjoint from the invariant circle $\partial K \times \{0\}$.
\end{lem}

\begin{prop} \label{prop-the-big-one}
Let $K \subset {\mathbb H}^2$ be a convex set with piecewise $C^2$-smooth boundary.
Assume that the curvature  of the boundary
has a jump discontinuity point $p \in \partial K$, where  
the one sided limits of the curvature are positive. 
Then there is an open neighborhood in the phase space $\Omega$
of the form $I \times [0,\varepsilon)$  that no invariant circle intersects,
where $I$ is some interval and $\varepsilon >0$.
\end{prop}

The proofs of Lemma~\ref{lem-step2} and Proposition~\ref{prop-the-big-one} appear after the proof of  Theorem~\ref{Thm_Hubbacher_Hyperbolic}. 

\noindent {\bf Proof of Theorem~\ref{Thm_Hubbacher_Hyperbolic}.}
Suppose, on the contrary, that every neighbourhood of $\partial K \times \{ 0 \}$ intersects some invariant circle. Thus we obtain a sequence of invariant circles $\Gamma_n$, whose distances to the boundary $\partial K \times \{ 0 \}$ are arbitrarily small, i.e.,
\begin{equation*}\label{eq-seq-touching-bd}
\mbox{dist}
\left(\partial K \times \{0\},\, \Gamma_n \right)
\to 0.
\end{equation*}
Using Bihkhoff's theorem mentioned above, these invariant circles correspond to a sequence of Lipshits continuous functions $f_n: \partial K \to [0,\pi]$ which all have the same Lipschitz constant.
By the Arzel\`a-Ascoli Theorem, we may assume, possibly passing to a subsequence, that $\{f_n\}$
converges uniformly to a function $f$. It is easy to check that the graph of $f$ is an invariant circle, which we denote by $\Gamma$.
On one hand, since $\{\Gamma_n\}$ approaches $\partial K \times \{ 0 \}$, the invariant circle $\Gamma$ must intersect $\partial K \times \{ 0 \}$. On the other hand, the circle $\Gamma_n$ do not intersect $I \times [0, \varepsilon]$, for the interval $I$ obtained in Proposition~\ref{prop-the-big-one} above, and thus $\Gamma$ does not coincide with $\partial K \times \{ 0 \}$. This is prohibited by Lemma~\ref{lem-step2}, which completes the proof of the theorem.
\qed


\noindent {\bf Proof of Lemma~\ref{lem-step2}.}
Suppose, by contradiction, that an invariant circle
$\Gamma$ intersects, but does not coincide with, $\partial K \times \{0\}$. Note that
$\Gamma$ encloses some open set $W$ that is homeomorphic to a
disk (see Figure \ref{fig:disjoint}). Since the billiard map $\phi$ is the identity on
$\partial K \times \{0\}$, $W$ is invariant under $\phi$. A
vertical line $\ell$ passing through $W$ divides it into two open
sets, $W_L$ and $W_R$, (to the left and right of $\ell$ respectively).
The monotone twist condition (see~\eqref{eq-mon-twist-cond} above) implies that the image of $\ell$ under
$\phi$ `bends to the right'. This means that
$\phi(W_R) \subsetneqq W_R$, and in particular $\phi(W_R)$ has
smaller area than $W_R$, which contradicts the area preserving
property of $\phi$.
\qed

\begin{figure}[h]
    \centering
    \includegraphics{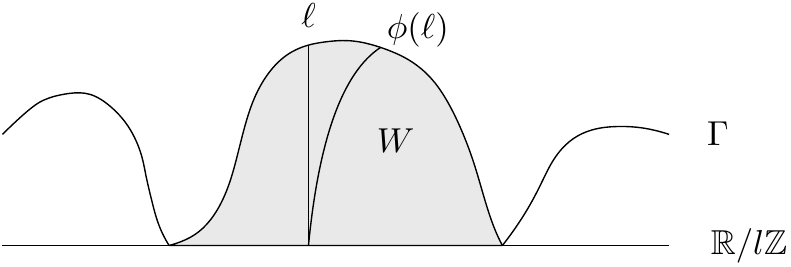}
    \caption{An invariant circle intersecting the boundary curve}
    \label{fig:disjoint}
\end{figure}

In what follows, we use the well known ``no crossing'' property of invariant circles of monotone twist maps.

\begin{lem}\label{lem-no-crossing}
Let $\Gamma \subset {\mathbb R}/l{\mathbb Z} \times [0,\pi] $ be an invariant circle of a monotone twist map as above, and let
$O  = \left\{ ( s_k, t_k) \right\}_{k\in\Z}\,$ and
$O' = \left\{ (s'_k,t'_k) \right\}_{k\in\Z}$ be two orbits lying on $\Gamma$. Then $O$ and $O'$ cannot cross, i.e. for all $k \in {\mathbb Z}$:
\[
s'_0\in (s_0,s_1)
\quad \Rightarrow \quad
s'_k\in (s_k,s_{k+1}).
\]
\end{lem}
\noindent {\bf Proof of Lemma~\ref{lem-no-crossing}.}
Since $\phi|_\Gamma$ corresponds to a homeomorphism of $S^1$, its lift $f_\Gamma :\R \to \R$ is a bijective monotone function.
The proof follows from the fact that $s'_k = f_\Gamma^k (s'_0)$, and $s_k = f_\Gamma^k (s_0)$.\qed

\noindent {\bf Proof of Proposition~\ref{prop-the-big-one}.}
{Denote by $\kappa_0$ and $\kappa_1$ the (positive) one-sided limits of the curvature  at $p \in \partial K$ (from the left and the right, respectively), and assume without loss of generality that $\kappa_1<\kappa_0$. Set the arclength parameter $s\in\R$ such that at the point $p \in \partial K$ one has $s=0$.
}
Consider the function $\tau(s)$ defined on (a subset of) the boundary $\partial K$ as follows. The value $\tau(s) \in [0,\pi]$ is the unique angle such that the line corresponding to $(s, \tau(s))$  is orthogonal to the normal line at $p$ (see Figure \ref{fig-Delta}). 
	\begin{figure}[h]
	\centering
	\includegraphics[scale=0.8]{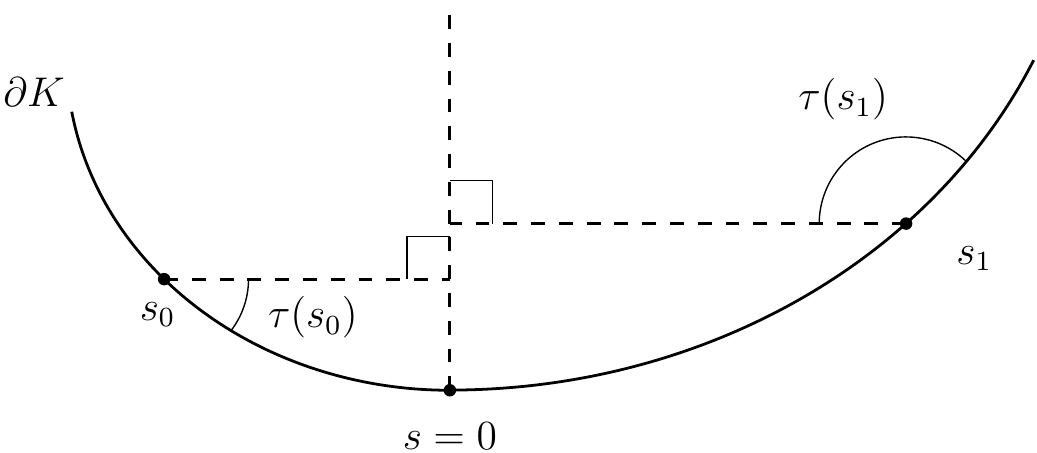}
	\caption{The function $\tau$}\label{fig-Delta}
\end{figure}

Note that $\tau(s)$ is well defined near the point $p$ (where $s=0$). Moreover, it is strictly decreasing as $s\to0^-$, and $\lim\limits_{s \to 0^-}\tau(s)=0$, as follows, e.g., from the Gauss-Bonnet formula (for $s_1 < s_2 \le 0$):
%
\begin{equation*}
\tau(s_1)-\tau(s_2) =  \int_{s_1}^{s_2} \kappa(\sigma) d\sigma
- (s_2-s_1)\overline o(s_1).
\end{equation*}

Next, consider a point $(s_0, t_0)$ with $t_0=\tau (s_0)$, and $s_0<0$. Then, for $(s_1,t_1) =\phi(s_0,t_0) $ one has $s_0 < 0 < s_1$.
We first show that if $(s_0,t_0)$ is chosen inside a sufficiently small neighborhood $U \subset \Omega$ containing $(0,0)$, then one has
$t_1 < (1-2\delta) t_0 < t_0$,
for some $\delta\in \left(0,\frac12 \right)$. That is
\begin{equation} \label{eq-first-req-from-U}
	 t_0 - t_1 > 2\delta t_0.
\end{equation}
Indeed, consider the left-sided and right-sided osculating \emph{curves of constant curvature} to $\partial K$ at the point $s=0$, with constant curvatures $\kappa_0$ and $\kappa_1$, respectively. We recall (see \cite{GR}) that in the hyperbolic plane there are three types of curves of constant positive curvature: hyperbolic circles (with $\kappa > 1$), horocycles (with $\kappa = 1$), and equidistant curves, i.e. curves lying at a fixed distance from a given geodesic (with $0 < \kappa < 1$). Denote the angles that they make with the chord from $s_0$ to $s_1$ by $\alpha_0$ and $\alpha_1$, respectively, and the distance of that chord from the point $s=0$ by $h$ (see Figure \ref{fig:osculating}). 
\begin{figure}[h]
    \centering
    \includegraphics[scale=0.8]{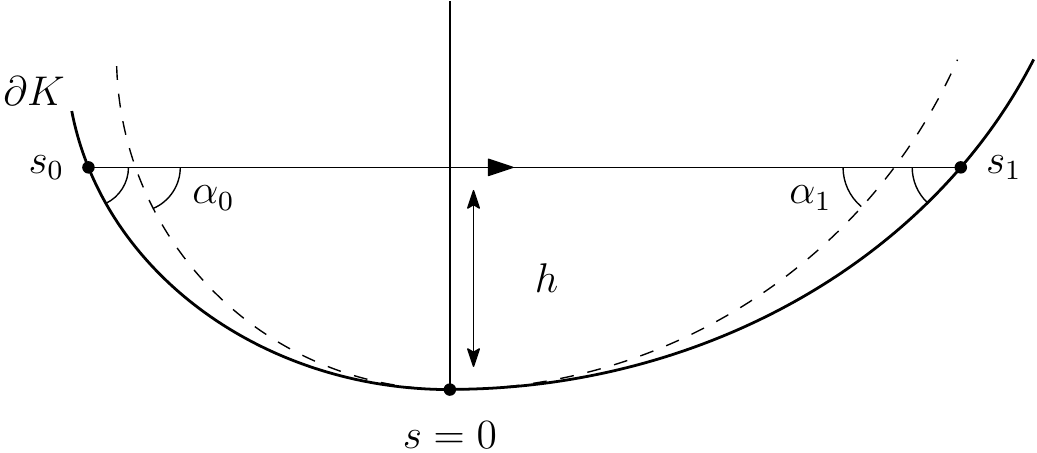}
    \caption{The osculating curves}
    \label{fig:osculating}
\end{figure}
A simple hyperbolic geometry exercise shows that these are related by 
\begin{equation}\label{eq:alpha_h}
\alpha_j = f_{\kappa_j} (h), \qquad (j=0,1)
\end{equation}
where the function $f_\kappa$ is defined by 
\begin{equation*}
f_\kappa(h) = \arccos{\bigl( \cosh(h) - \kappa \sinh(h) \bigr)}.
\end{equation*}
For example, \eqref{eq:alpha_h} is illustrated in Figure \ref{fig:hyperbolic_exercise} for the case of a hyperbolic circle of radius $R$ and hence curvature $\kappa = \coth(R)$. 
\begin{figure}[h]
    \centering
    \includegraphics[width=0.55\linewidth]{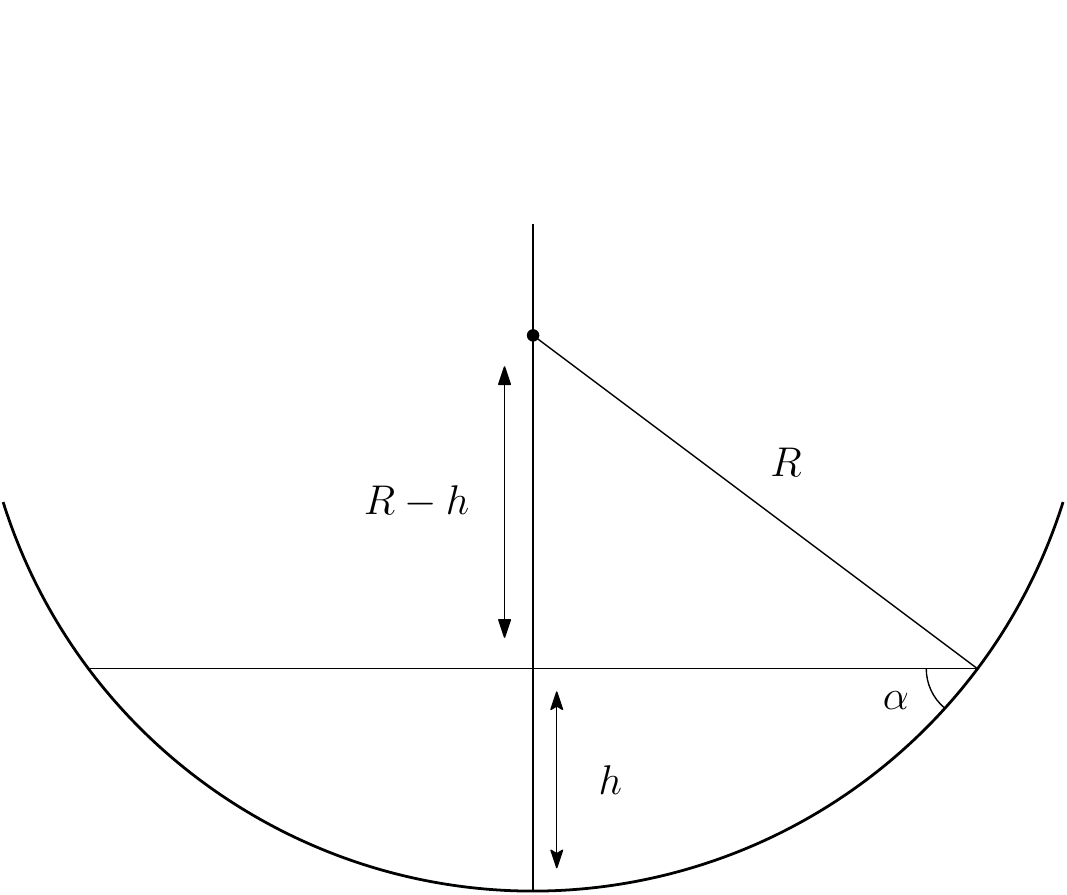}
    \caption{For a circle, the relation $\alpha = f_\kappa(h)$ follows from the hyperbolic law of sines: $\cos \alpha = \frac{\sinh(R-h)}{\sinh (R)}$}
    \label{fig:hyperbolic_exercise}
\end{figure}
Finally, using the expansion
\begin{equation*}
\arccos(1-x) = \sqrt{2x} + \overline{o}(\sqrt{x})
\end{equation*}
we deduce that 
\begin{equation}\label{eq:f_k_exp}
f_\kappa(h) = \sqrt{2\kappa h} + \overline{o}(\sqrt{h}).
\end{equation}
From \eqref{eq:alpha_h} and \eqref{eq:f_k_exp} it follows easily that
\begin{equation}\label{eq:alpha_taylor}
\alpha_1 = \sqrt{\frac{\kappa_1}{\kappa_0}} \,\alpha_0 + \overline{o}(\alpha_0).
\end{equation}


Using the second-order approximation of the boundary $\partial K$ by the osculating curves we deduce from \eqref{eq:alpha_taylor} that 
$$t_1 = \sqrt{\frac{\kappa_1}{\kappa_0}}\,t_0 + \overline{o}(t_0).$$
Thus, by choosing the neighborhood $U$ to be sufficiently small and using the assumption $\kappa_1 < \kappa_0$, we conclude that $t_1 < \mu t_0$ for some $\mu \in (0,1)$, and inequality \eqref{eq-first-req-from-U} follows, for $\delta = \frac12 (1-\mu)$.

Next, we shrink $U$ if necessary, so that on both $U_{+}:=U \cap \{s > 0\}$ and $U_{-}:=U \cap \{s < 0\}$ one has the following approximation for the billiard  map (which follows, e.g., from~\cite[Lemma 8]{CP}, combined with a limiting argument when $t\to 0^+$):
\begin{equation}\label{eq-taylor}
\phi(s,t)
=
\begin{pmatrix}
1 & \frac{2}{\kappa(s)}\\
0 &       1    
\end{pmatrix}
\begin{pmatrix}
s\\
t
\end{pmatrix}
+ \overline o(t),
\end{equation}
whenever $(s,t)$ and $\phi(s,t)$ are either both in $U_{+}$ or both in $U_{-}$. Finally, since the curvature is continuous from either side of the point $p$, we may further shrink $U$ so that the bounds $m_-, m_+, M_-, M_+$ defined by
\[
M_- = \sup_{U \cap \{s < 0\}} \left\{\frac{2}{\kappa(s)}\right\}, \quad
m_- = \inf_{U \cap \{s < 0\}} \left\{\frac{2}{\kappa(s)}\right\},
\]
\[
M_+ = \sup_{U \cap \{s > 0\}} \left\{\frac{2}{\kappa(s)}\right\}, \quad
m_+ = \inf_{U \cap \{s > 0\}} \left\{\frac{2}{\kappa(s)}\right\},
\]
satisfy 
$\frac{m_-}{M_-}, \frac{m_+}{M_+} \in
\left(\frac{1}{1+\delta}, 1 \right)$, which ensures that
\begin{equation*}
(1-\delta)M_-  < m_-
\end{equation*}
\begin{equation*}
 M_+ < m_+(1+\delta).
\end{equation*}
Having chosen the neighborhood $U$ as above, we next choose a sufficiently small rectangular neighborhood  $\widetilde U \subset U$ containing $(0,0)$, in a way which guarantees that starting at $(s,t) \in \widetilde U$, the billiard trajectory (both forward and backward) remains inside $U$ for $n$ consecutive reflections, where
 $n\in \N$ is defined by $$n = \max
        \left\{
        \Bigl\lceil \frac{m_-}{m_--(1-\delta)M_-} \Bigr\rceil,
        \Bigl\lceil \frac{(1+\delta)m_+}{m_+(1+\delta) -M_+} \Bigr\rceil
        \right\}.$$
Note that this choice of $n$ implies that:
\begin{equation}\label{eq-cond-nmM-case1}
    n(1-\delta)M_- \le (n-1)m_-,
\end{equation}
\begin{equation}\label{eq-cond-nmM-case2}
    nM_+ \le (n-1)m_+(1+\delta).
\end{equation}
Consider the intersection point $(a,\tau(a))$ of the graph of $\tau$ with the boundary $\partial \widetilde U$, where $a<0$. 
Let $b\in(a,0)$, and let 
 $V = [a,b] \times [0, \tau(b)]$ be a rectangle inside $\widetilde U$ under the graph of $\tau$ (see Figure \ref{fig:V}).
 \begin{figure}[h]
     \centering
     \includegraphics[width=.40
     \textwidth]{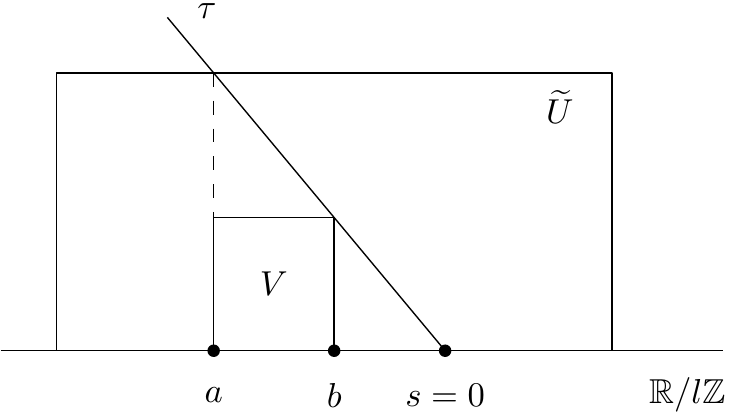}
     \caption{The neighbourhood $V$}
     \label{fig:V}
 \end{figure}
 
Since an invariant circle $\Gamma$ is a Lipschitz curve, and by Birkhoff's theorem one has a uniform bound on the Lipschitz constant of any such circle, the neighborhood $V$ can be further shrunk to a rectangle $I\times [0,\varepsilon]$ so that if $\Gamma$ intersects $V$, then $(0,\Gamma(0))$ lies in $\widetilde U$.
We will show that no invariant circle intersects $V = I \times [0,\varepsilon]$.
Assume by contradiction that $\Gamma$ is an invariant circle passing through $V$. 
Note that, by the specific choice of $V$, the curves $\Gamma$ and $\tau$ intersect at a point $(s_0,t_0) \in \widetilde U$, and $(s'_0,t'_0) := (0,\Gamma(0)) \in \widetilde{U}$ as well. 
The main idea of the proof is to show that the jump in the curvature implies that the two orbits 
$O  = \left\{ ( s_k, t_k) \right\}_{k\in\Z}$, and
$O' = \left\{ (s'_k,t'_k) \right\}_{k\in\Z}$ lying on the invariant curve $\Gamma$ must cross, in contradiction to the monotonicity of $\phi|_\Gamma$, stated in Lemma \ref{lem-no-crossing} above.



Note that, by~\eqref{eq-first-req-from-U}, one has $(1+\delta)t_1 < (1-\delta)t_0$, which implies that either $t_0' < (1-\delta) t_0$,\, or $\,(1+\delta) t_1 < t_0'$. We consider these two cases separately, and exhibit  in each of them, a forbidden crossing (within $n$ reflections),  thus reaching the desired contradiction.
More precisely, since $0=s'_0 \in (s_0, s_1)$, one has, by Lemma \ref{lem-no-crossing}, that for all $k\in\mathbb{Z}$
\begin{equation}\label{eq-no-crossing-cond}
s'_k \in (s_k, s_{k+1}).
\end{equation}

\noindent {\bf Case 1.} Assume $t'_0 < (1-\delta) t_0$.
In this case we will obtain a crossing for a negative index, that is
$s'_{-n}\notin (s_{-n},s_{-n+1})$. 
Recall that the points $\left\{ ( s_k, t_k) \right\}_{k = -n}^{k=0},\,
\left\{ (s'_k,t'_k) \right\}_{k = -n}^{k=0}$ all remain inside
$U_-$. By~\eqref{eq-taylor}, one has
\[
\begin{array}{ccc}
s_{k+1} = & s_k ~~+ & \frac{2}{\kappa(s_k)}~t_k + \overline{o}(t_k),\\
t_{k+1} = &       &                  ~~~~~~~t_k + \overline{o}(t_k),
\end{array}
\]
for $k\in \{-n+1,\dots,-1 \}$. Since $n$ is fixed, we may equivalently write this as
\[
\begin{array}{ccc}
s_{k+1} = & s_k ~~+ & \frac{2}{\kappa(s_k)}~t_0 + \overline{o}(t_0),\\
t_{k+1} = &         &                ~~~~~~~~t_0 + \overline{o}(t_0).
\end{array}
\]
Similarly, for $k\in \{-n,\dots,-1 \}$ one has:
\[
\begin{array}{ccc}
s'_{k+1} = & s'_k ~~+ & \frac{2}{\kappa(s'_k)}~t'_0 + \overline{o}(t'_0)\\
t'_{k+1} = &          &                 ~~~~~~~~t'_0 + \overline{o}(t'_0).
\end{array}
\]
Since the billiard trajectories remain inside $U_-$, one has ${\frac{2}{\kappa(s_k)}}, {\frac{2}{\kappa(s'_k)}} \in [m_-,M_-]$, so
\begin{eqnarray*}\label{eq-crossing-negative-n-case1}
       - s'_{-n} =
s'_{0} - s'_{-n} &=& 
\sum_{k=-n}^{-1} s'_{k+1} - s'_k=
\left( \sum_{k=-n}^{-1} \frac{2}{\kappa(s'_k)} \right) t'_0 + \overline{o}(t'_0) 
  <  n M_- (1-\delta) t_0 + \overline{o}(t_0), \\
       - s_{-n+1} >
s_{0} - s_{-n+1} &=&
\sum_{k=-n+1}^{-1} s_{k+1} - s_k =
\left( \sum_{k=-n+1}^{-1} \frac{2}{\kappa(s_k)} \right) t_0 + \overline{o}(t_0) \ge
(n-1)m_- t_0 + \overline{o}(t_0).
\end{eqnarray*}
From \eqref{eq-cond-nmM-case1} it follows that, by shrinking the neighborhood $U$ further (before the choice of $\widetilde U$), we get $s_{-n+1} < s'_{-n}$, thus violating
$s'_k\in (s_k,s_{k+1})$ for $k=-n$. The second case is handled similarly. We provide the details for completeness.\\ \\
\noindent {\bf Case 2.} Assume $(1+\delta) t_1 < t'_0$.
In this case we will obtain a crossing for a positive index, that is $s'_n\notin (s_n,s_{n+1})$.
Note that, by \eqref{eq-no-crossing-cond}, one has $s_1 \le s'_1$.
Since the points $\left\{ ( s_k, t_k) \right\}_{k = 1}^{k=n+1},\,
\left\{ (s'_k,t'_k) \right\}_{k = 1}^{k=n}$ all remain inside
$U_+$, we have, as before, for $k\in \{1,\dots,n \}$
\[
\begin{array}{ccc}
s_{k+1} = & s_k ~~+ & \frac{2}{\kappa(s_k)}~t_1 + \overline{o}(t_1),\\
t_{k+1} = &         &                ~~~~~~~~t_1 + \overline{o}(t_1).
\end{array}
\]
Similarly, for $k\in \{1,\dots, n-1 \}$ we have
\[
\begin{array}{ccc}
s'_{k+1} = & s'_k ~~+ & \frac{2}{\kappa(s'_k)}~t'_0 + \overline{o}(t'_0),\\
t'_{k+1} = &          &                 ~~~~~~~~t'_0 + \overline{o}(t'_0),
\end{array}
\]
and consequently 
\begin{eqnarray*}\label{eq-crossing-negative-n-case2}
s'_{n} - s'_{1} &=& 
\sum_{k=1}^{n-1} s'_{k+1} - s'_{k} =
\left( \sum_{k=1}^{n-1} \frac{2}{\kappa(s'_k)} \right) t'_0 + \overline{o}(t'_0) >
 (n-1) m_+ (1+\delta)t_1 + \overline{o}(t_1), \\
s_{n+1} - s'_{1} \le
s_{n+1} - s_1 &=&
\sum_{k=1}^{n} s_{k+1} - s_k =
\left( \sum_{k=1}^{n} \frac{2}{\kappa(s_k)} \right) t_1 + \overline{o}(t_1) \le
n M_+ t_1 + \overline{o}(t_1).
\end{eqnarray*}
After shrinking $U$ as before, by 
\eqref{eq-cond-nmM-case2} we get $s_{n+1} < s'_{n}$,
thus violating $s'_k\in (s_k,s_{k+1})$ for $k=n$.


In both cases we obtained a crossing, contradicting \eqref{eq-no-crossing-cond}, which implies the invariant circle $\Gamma$ could not have intersected $V = I \times [0,\varepsilon]$, thus completing the proof of the proposition.
\qed

\begin{rem}\label{rem-spheric-fig-four} {\rm
In the case of $\mathbb S^2_+$, the difference arises in the analysis of Figure \ref{fig:osculating}, where one only has to consider osculating circles, and Equation \eqref{eq:alpha_taylor} is obtained using the spherical law of sines, similarly to the proof sketched in Figure \ref{fig:hyperbolic_exercise}.}
\end{rem}

\noindent Dan Itzhak Florentin \\
\noindent Department of Mathematics, Bar Ilan University, Israel \\
\noindent e-mail: dan.florentin@biu.ac.il
 \vskip 10pt
\noindent Yaron Ostrover \\
\noindent School of Mathematical Sciences, Tel Aviv University, Israel \\
\noindent e-mail: ostrover@tauex.tau.ac.il
\vskip 10pt
\noindent Daniel Rosen \\
\noindent Faculty of Mathematics, Ruhr-Universit\"at Bochum, Germany \\
\noindent e-mail: daniel.rosen@rub.de

\end{document}